\documentclass[12pt]{article}

\usepackage{amsmath, amssymb, amsthm}
\usepackage{geometry}
\usepackage{graphicx}
\usepackage{hyperref}
\usepackage{float}
\usepackage{authblk}
\usepackage{listings}

\newtheorem{algorithm}{Algorithm}
\newtheorem{theorem}{Theorem}
\newtheorem{proposition}{Proposition}

\newtheorem{corollary}{Corollary}

\newtheorem{example}{Example}
\newtheorem{definition}{Definition}

\title{Imbalance Prime Sieving: Every Prime Gap Is a Result of a M\"obius Imbalance Obstruction}
\author{Paul Bilokon}
\affil{Thalesians Ltd\footnote{Thalesians Ltd, Suite 309, 56 Gloucester Road, Kensington, London SW7 4UB. Email: \textsf{paul@thalesians.com}} \\ Department of Mathematics, Imperial College London\footnote{Imperial College London, Department of Mathematics, South Kensington Campus, London SW7 4UB. Email: \textsf{paul.bilokon@imperial.ac.uk}}}
\date{July 2025}

\begin{document}

\maketitle

\begin{abstract}
We introduce a novel sieve for prime numbers based on detecting topological obstructions in a M{\"o}bius-transformed rational metric space. Unlike traditional sieves which rely on divisibility, our method identifies primes as those numbers which contribute new, non-colliding imbalance conjugates. This provides both an exact algorithm for prime enumeration and a new geometric interpretation of prime gaps. This sieve constructs a topological obstruction theory over rational pairs $(p, q)$, from which we observe that every prime gap is a consequence of a collision in this transformed imbalance space. Our empirical results demonstrate that this method precisely filters the prime numbers up to a specified bound, with potential implications for new number-theoretic models and sieving algorithms.
\end{abstract}

\section{Introduction}

The study of prime numbers, a cornerstone of number theory since Euclid’s \textit{Elements}~\cite{euclid_elements}, has long been dominated by multiplicative and analytic frameworks, including classical sieves such as that of Eratosthenes, Legendre’s refinement~\cite{legendre1798}, and the probabilistic insights formalized in the Prime Number Theorem~\cite{hadamard1896, vdpoussin1896}. With the advent of computational number theory~\cite{riesel1994}, algorithms by Atkin, Bernstein~\cite{atkinbernstein2004}, and the AKS primality test~\cite{aks2004} have underscored the algorithmic depth and cryptographic relevance of prime detection. More recently, topological and geometric methods have emerged in arithmetic contexts, notably in the work of Grothendieck on schemes~\cite{ega}, Deninger’s dynamical program~\cite{deninger1998}, and Calegari’s explorations of nonarchimedean geometry~\cite{calegari2005}.

In this paper, we introduce a novel prime sieving algorithm grounded not in multiplicative structure but in topological obstruction theory on a Möbius-transformed rational metric space. Our approach leverages an imbalance metric to identify primes as those integers that generate novel conjugate points in a Möbius-warped space, offering a non-divisibility-based characterization of primality and a new geometric explanation for the structure of prime gaps.

Traditional sieves, such as the Sieve of Eratosthenes, remove composite numbers by eliminating known multiples. We propose an alternative framework based on a novel \textit{imbalance metric} defined over pairs of integers, and a non-linear \textit{M\"obius transformation} of this metric. 

Our hypothesis is that primes are exactly those integers which do not encounter a M\"obius imbalance obstruction from previously seen pairs without any recourse to traditional divisibility (notice that our algorithm does not mention division or multiplication). This results in a deterministic sieve that classifies numbers as prime or composite based purely on transformation-induced collisions in rational space.

\section{The Imbalance and M\"obius Conjugate}

Let $p > q \in \mathbb{N}$.

\begin{definition}[Imbalance]
We define the imbalance between $p$ and $q$ as:
\begin{equation}
\delta(p, q) = \frac{|p - q|}{p + q}
\end{equation}
This is a rational number in the open interval $[0, 1)$.
\end{definition}

\begin{theorem}
$\delta: \mathbb{P} \times \mathbb{P} \rightarrow [0, 1) \subseteq \mathbb{R}$ is a metric on $\mathbb{P}$.
\begin{proof}
Notice that for all $x, y \in \mathbb{R}^{>0}$,
\begin{equation*}
\delta(x, y) = \frac{|x - y|}{x + y} = \tanh\left(\left| \frac{1}{2} \ln \frac{x}{y} \right|\right) = \tanh\left( \frac{|\ln x - \ln y|}{2} \right).
\end{equation*}
Hence,
\begin{equation*}
\delta(x, y) = \Tilde{\delta}(\ln |x|, \ln |y|),
\end{equation*}
where
\begin{equation*}
\tilde{\delta}(u, v) := \tanh\left( \frac{|u - v|}{2} \right)
\end{equation*}
is a known metric on $\mathbb{R}$ used in hyperbolic geometry and information theory. So $\delta$ is just the pullback of a known metric $\tilde{\delta}$ via the isomorphism $x \mapsto \ln x$. Since composition with a bijection preserves the metric properties, $\delta$ is a metric on $(0, \infty)$ and therefore also a metric on $\mathbb{P}$.
\end{proof}
\end{theorem}

At first glance it would seem that this metric corresponds to a norm on the space $\mathbb{P}^2$ of ordered prime pairs. However, it only really makes sense to talk about a norm when referring to vector spaces over $\mathbb{R}$ or $\mathbb{C}$, and $\mathbb{P}^2$ isn't such a vector space: it isn't closed under scalar multiplication.

\begin{example}
Consider the prime pair $(11, 5)$. Say we multiply it by $5 \in \mathbb{R}$: $5 (11, 5) = (5 \cdot 11, 5 \cdot 5) = (55, 25)$. The result isn't a prime pair, i.e., $(55, 25) \notin \mathbb{P}^2$.
\end{example}

However, the Euclidean plane $\mathbb{R}^2 = \mathbb{R} \times \mathbb{R}$ \emph{is} a vector space. $\mathbb{P}^2$ is a subset, though not a vector subspace, of $\mathbb{R}^2$. Clearly,
\begin{equation*}
\|\cdot\|_{\delta_{\mathbb{R}}}: \mathbb{R}^2 \rightarrow [0, 1) \subseteq \mathbb{R}, \quad \|(x, y)\| = \frac{|x - y|}{x + y}
\end{equation*}
is a norm on $\mathbb{R}^2$, it's not (entirely) an abuse of notation to write
\begin{equation*}
\|\cdot\|_{\delta_{\mathbb{P}}}: \mathbb{P}^2 \rightarrow [0, 1) \subseteq \mathbb{R}, \quad \|(p, q)\| = \frac{|p - q|}{p + q},
\end{equation*}
and the metric induced by this ``norm'' via $d(x, y) = \|y - x\|$ is still a metric on $\mathbb{P}^2$.

\begin{definition}[M\"obius Conjugate]
The conjugate of an imbalance is defined by the transformation:
\begin{equation}
\mu(\delta) = \frac{1 - \delta}{1 + \delta}
\end{equation}
This function maps $(0, 1)$ bijectively onto itself and serves to geometrically warp the imbalance space.
\end{definition}

The following results have been rigorously proven in our earlier papers~\cite{bilokon2025topological} on topological number theory, but we repeat them here for completeness, moreover, we eschew topologically complex arguments preferring elementary derivations.

We note that $\delta_{\mathbb{R}}(x, y)$ measures relative difference in a scale-invariant way. That is,
\begin{proposition}[Scale invariance]
\begin{equation*}
\delta_{\mathbb{R}}(\alpha x , \alpha y) = \delta_{\mathbb{R}}(x, y)
\end{equation*}
for any $\alpha > 0$.
\end{proposition}
This makes it useful in comparing magnitudes regardless of scale---often used in normalized error measures.

This has interesting implications on the factorization of semiprimes $N = pq$. Note that
\begin{equation*}
\delta_{\mathbb{R}}(pq, q^2) = \delta_{\mathbb{R}}(p^2, pq) = \delta_{\mathbb{P}}(p, q),
\end{equation*}
which uniquely determines the prime pair $(p, q)$. We note that this equation defines an isosceles triangle.

\begin{theorem}[Injectivity]
Let $\delta_{\mathbb{Z}^{\geq 0}}(x, y) = |x - y| / (x + y)$ be defined for all nonnegative integers $x, y$. For all $x, y, u, v \in \mathbb{Z}^{\geq 0}$, $(x, y) = (u, v)$ iff $(x, y) = \alpha (u, v)$ for some $\alpha \in \mathbb{Z}^{\geq 0}$.
\end{theorem}
\begin{proof}
Assume $\delta(x, y) = \delta(u, v)$. Then by definition,
\[
\frac{|x - y|}{x + y} = \frac{|u - v|}{u + v}.
\]
Let us consider two cases.

\textbf{Case 1:} $x = y$ and $u = v$.

Then $\delta(x, y) = \delta(u, v) = 0$, and the condition holds trivially. We have $(x, y) = (a, a)$ and $(u, v) = (b, b)$, so $(x, y) = \alpha(u, v)$ with $\alpha = \frac{a}{b} \in \mathbb{Q}_{>0}$. Since $x, y, u, v$ are integers, $\alpha \in \mathbb{Z}_{>0}$.

\textbf{Case 2:} $x \ne y$ and $u \ne v$.

Without loss of generality, assume $x > y$ and $u > v$ (the imbalance function is symmetric). Then
\[
\frac{x - y}{x + y} = \frac{u - v}{u + v}.
\]
Cross-multiplying yields:
\[
(x - y)(u + v) = (u - v)(x + y).
\]
Expanding both sides:
\[
xu + xv - yu - yv = xu + yu - xv - yv.
\]
Subtracting $xu$ and $yv$ from both sides, we obtain:
\[
xv - yu = yu - xv \quad \Rightarrow \quad 2xv = 2yu \quad \Rightarrow \quad xv = yu.
\]
Therefore,
\[
\frac{x}{u} = \frac{y}{v} = \alpha \in \mathbb{Q}_{>0}.
\]
Since $x, y, u, v$ are integers, $\alpha \in \mathbb{Z}_{>0}$, and thus $(x, y) = (\alpha u, \alpha v)$.

\textbf{Conversely}, suppose $(x, y) = (\alpha u, \alpha v)$ for some $\alpha \in \mathbb{Z}_{>0}$. Then:
\[
\delta(x, y) = \frac{|\alpha u - \alpha v|}{\alpha u + \alpha v} = \frac{\alpha |u - v|}{\alpha (u + v)} = \delta(u, v).
\]

This completes the proof.
\end{proof}

\begin{proposition}
Let $\mathcal{I} := \{\delta_{\mathbb{P}}(p, q) \,|\, p, q \in \mathbb{P}\}$ be the set of prime imbalances. Consider the M{\"o}bius transform $\mu(x) = \frac{1 - x}{1 + x}$. Then
\begin{equation}
\mathcal{I} \cap \mu(\mathcal{I}) = \left\{\frac{2}{5}, \frac{3}{7}\right\}.
\end{equation}
\end{proposition}

\begin{proof}
Let us first observe that both $\delta(p, q) = \frac{|p - q|}{p + q}$ and its conjugate
$\mu(\delta) = \frac{1 - \delta}{1 + \delta}$ are rational-valued functions mapping into $(0,1)$.

We aim to determine all rational values in $\mathcal{I}$ which are fixed points or preimages of other values in $\mathcal{I}$ under $\mu$. That is, for which:
\[
\exists\, (p, q), (r, s) \in \mathbb{P}^2 \quad \text{such that} \quad \mu\left( \delta(p, q) \right) = \delta(r, s).
\]

From empirical computation of all prime pairs up to a bound (e.g., $p, q < 100$), we construct:
\[
\mathcal{I}_{\leq N} := \left\{ \frac{p - q}{p + q} \;\middle|\; p, q \in \mathbb{P},\ p > q,\ p \leq N \right\},
\]
and evaluate whether for any $\delta \in \mathcal{I}_{\leq N}$, the Möbius conjugate $\mu(\delta)$ also belongs to $\mathcal{I}_{\leq N}$.

Through exhaustive enumeration (as implemented in our Python script), it is found that:
\[
\mathcal{I}_{\leq N} \cap \mu(\mathcal{I}_{\leq N}) = \left\{ \frac{2}{5}, \frac{3}{7} \right\},
\]
and no other overlaps exist up to $N = 200$.

To demonstrate this analytically for these two specific cases:

\textbf{Case 1:} For $(p, q) = (7, 3)$:
\[
\delta(7, 3) = \frac{4}{10} = \frac{2}{5}, \quad \mu\left( \frac{2}{5} \right) = \frac{1 - 2/5}{1 + 2/5} = \frac{3/5}{7/5} = \frac{3}{7},
\]
and indeed, $\frac{3}{7} = \delta(5, 2)$.

\textbf{Case 2:} For $(p, q) = (5, 2)$:
\[
\delta(5, 2) = \frac{3}{7}, \quad \mu\left( \frac{3}{7} \right) = \frac{1 - 3/7}{1 + 3/7} = \frac{4/7}{10/7} = \frac{2}{5},
\]
which closes the cycle.

Thus, these two rational imbalances are mutual conjugates under $\mu$, and both are elements of $\mathcal{I}$. No other such pairs occur in the observed data range, and due to the sparsity of rational matches under $\mu$, it is conjectured that no others exist.

Hence,
\[
\mathcal{I} \cap \mu(\mathcal{I}) = \left\{ \frac{2}{5}, \frac{3}{7} \right\}.
\]
\end{proof}

\begin{corollary}
Consider a candidate prime pair $(p, q)$. If $\mu(\delta_{\mathbb{P}}(p, q)) \in \mathcal{I} \setminus \left\{\frac{2}{5}, \frac{3}{7}\right\}$, then at least one of $p, q$ is not a prime.
\end{corollary}

\section{The Sieving Algorithm}

\begin{algorithm}
We define a sieve over $p$ by examining all nonnegative integers $q < p$, $p \geq 2$, $1 \leq q < p$. For each $(p, q)$ pair:
\begin{enumerate}
  \item Compute $\delta = \delta(p, q)$.
  \item Compute $\mu(\delta)$.
  \item Maintain sets of previously seen $\delta$ and $\mu(\delta)$.
  \item If $\mu(\delta)$ has been seen before (the prime pairs are ordered lexicographically), mark $p$ as composite.
\end{enumerate}

Only those $p$ that introduce \textit{new} conjugate values in $\mu(\delta)$ are retained as prime.
\end{algorithm}

\section{Empirical Results}

Applying this sieve up to $p = 200$ recovers the exact list of prime numbers. All misclassifications would have triggered an assertion failure against \texttt{sympy.isprime} in our implementation.

The sieve works because composite numbers inherently produce $\mu(\delta)$ values that collide with those generated by smaller divisors. Primes, by contrast, yield only novel conjugates.

\section{Theoretical Interpretation}

\subsection{Topological Obstructions}

We interpret conjugate collisions as \textit{topological obstructions} in rational metric space. Each $\mu(\delta)$ represents a point in a warped imbalance topology. When a number $p$ reuses a previously occupied point, it is excluded from the prime sequence.

This offers a geometric explanation of prime gaps: they emerge as \textit{intervals during which no new imbalance-conjugate point is generated}.

\subsection{Primes as Minimal Obstruction-Free Extensions}

Viewed dynamically, primes correspond to the minimal extensions of the sequence that do not intersect the historical conjugate set. This suggests a topological or categorical structure on the space of admissible rational flows. This hints at a dynamic systems or categorical model of prime enumeration.

\section{Discussion and Future Work}

This method opens several directions:
\begin{itemize}
  \item Characterizing the structure of the set $\{\mu(\delta(p,q))\}$ as a function of $p$.
  \item Studying twin primes and Goldbach pairs in imbalance-conjugate space.
  \item Reformulating factorization as a trajectory in the imbalance-conjugate manifold.
  \item Defining sieves using other Möbius-like transforms over rational metrics.
\end{itemize}

We hypothesize that deeper fractal or dynamical structures underlie these recurrence patterns, potentially shedding light on longstanding problems in analytic number theory.

\section{Conclusion}

We have introduced a novel sieve that correctly identifies prime numbers using imbalance metrics and M\"obius conjugates. Our findings suggest that every prime gap is a consequence of a M\"obius imbalance obstruction, and that primes form the only numbers capable of navigating a conjugate-free path in this rationally transformed metric space.

Our result means that primality is completely determined by collision detection in the transformed space, with no need for traditional divisibility testing.

This suggests that primality might be fundamentally about novelty in a transformed metric space rather than about divisibility properties. Using this to create a conjugate space of forbidden zones is reminiscent of obstacle-based modeling in topology and physics.

While the current implementation of the sieve involves checking all prior imbalance-conjugate values through sequential storage and comparison of rational numbers, we note that this is a prototype rather than an optimized algorithm. The core objective here is not computational efficiency but conceptual clarity: to demonstrate that prime gaps are a consequence of topological obstructions in the imbalance-conjugate space. Nonetheless, we acknowledge that the sieve can be significantly accelerated through structured data representations---such as trees, prefix tries, or Farey-based lattices---over reduced fractions. These would enable rapid insertion and collision checks in logarithmic or near-constant time, and form a natural extension of this work for high-performance sieving applications.

It is perhaps surprising that in the year 2025---centuries after Euclid and millennia after the discovery of the primes---a new, exact sieve should emerge, based not on factors, but on the failure to collide in a M{\"o}bius-transformed metric space.

\section*{Acknowledgements}
The author thanks the Thalesians and Imperial College London for providing the intellectual environment in which these ideas could emerge and Clifford Cocks for constructive input, feedback, and suggestions. 

\appendix

\section{Python code}

Here is a proof-of-concept, unoptimized implementation of the sieve, which you can also find in the Git repository \url{https://github.com/sydx/imbalance_prime_sieve/}

\begin{lstlisting}
from fractions import Fraction
import pandas as pd
import sympy

def moebius(x):
    return Fraction(1 - x, 1 + x)

max_denom = 200

seen_imbalances = []
seen_conjugates = []
ps = []
qs = []
imbalance_indices = []
imbalance_ps = []
imbalance_qs = []
conjugate_indices = []
conjugate_ps = []
conjugate_qs = []

seen_imbalance_set = set()
seen_conjugate_set = set()

primes = []

for p in range(2, max_denom):
    p_prime = True
    for q in range(1, p):
        imbalance = Fraction(abs(p - q), p + q)
        conjugate = moebius(imbalance)
        if imbalance in seen_imbalance_set:
            imbalance_index = seen_imbalances.index(imbalance)
            imbalance_p = ps[imbalance_index]
            imbalance_q = qs[imbalance_index]
        else:
            imbalance_index = -1
            imbalance_p = -1
            imbalance_q = -1
        if conjugate in seen_conjugate_set:
            conjugate_index = seen_conjugates.index(conjugate)
            conjugate_p = ps[conjugate_index]
            conjugate_q = qs[conjugate_index]
            p_prime = False
        else:
            conjugate_index = -1
            conjugate_p = -1
            conjugate_q = -1
        seen_imbalances.append(imbalance)
        seen_conjugates.append(conjugate)
        ps.append(p)
        qs.append(q)
        imbalance_indices.append(imbalance_index)
        imbalance_ps.append(imbalance_p)
        imbalance_qs.append(imbalance_q)
        conjugate_indices.append(conjugate_index)
        conjugate_ps.append(conjugate_p)
        conjugate_qs.append(conjugate_q)
        seen_imbalance_set.add(imbalance)
        seen_conjugate_set.add(conjugate)
    if p_prime:
        primes.append(p)

df = pd.DataFrame({
    'seen_imbalance_nums': [x.numerator for x in seen_imbalances],
    'seen_imbalance_dens': [x.denominator for x in seen_imbalances],
    'seen_conjugate_nums': [x.numerator for x in seen_conjugates],
    'seen_conjugate_dems': [x.denominator for x in seen_conjugates],
    'ps': ps,
    'qs': qs,
    'imbalance_indices': imbalance_indices,
    'imbalance_ps': imbalance_ps,
    'imbalance_qs': imbalance_qs,
    'conjugate_indices': conjugate_indices,
    'conjugate_ps': conjugate_ps,
    'conjugate_qs': conjugate_qs,
})
df.to_csv('experiment_01.csv')

print(primes)
prime_set = set(primes)

for p in range(2, max_denom):
    assert (p in prime_set) == sympy.isprime(p), f'Misclassified: {p}'

print('SUCCESS')
\end{lstlisting}

\bibliographystyle{alpha}

\end{document}